\documentclass[11pt]{amsart}

\usepackage{amsmath}
\usepackage{amsthm}
\usepackage{amsfonts}
\usepackage{amssymb} 

\usepackage{graphicx}
\usepackage{tikz}

\usepackage{caption}
\usepackage{subcaption}
\usepackage{chngcntr}

\newtheorem{theorem}{Theorem}[section]
\newtheorem{lemma}[theorem]{Lemma}

\theoremstyle{definition}
\newtheorem{definition}[theorem]{Definition}
\newtheorem{example}[theorem]{Example}

\theoremstyle{remark}

\newcommand{\M}{\mathcal{M}}
\newcommand{\N}{\mathbb{N}}

\newcommand{\BB}{{\mathcal B}}

\def\<{\left<}
\def\>{\right>}

\def\<{\left<}
\def\>{\right>}
\def\ll{\langle\kern-3pt\langle}
\def\rr{\rangle\kern-3pt\rangle}

\begin{document}


\title[The Dehn--Sommerville relations and the Catalan matroid]{The Dehn--Sommerville relations and\\ the Catalan matroid}


\author{Anastasia Chavez}
\address{Department of Mathematics, University of California, Berkeley, 970 Evans Hall 3840, Berkeley, CA 94720 USA.}
\curraddr{}
\email{a.chavez@berkeley.edu}
\thanks{Supported in part by NSF Alliances for Graduate Education and the Professoriate}

\author{Nicole Yamzon}
\address{Department of Mathematics, San Francisco State University, 1600 Holloway Ave, San Francisco, CA 94132, USA. }
\curraddr{}
\email{nyamzon@mail.sfsu.edu }
\thanks{Supported in part by NSF GK-12 grant DGE-0841164}

\date{}

\dedicatory{}

\begin{abstract}
The $f$-vector of a $d$-dimensional polytope $P$ stores the number of faces of each dimension. When $P$ is simplicial the Dehn--Sommerville relations condense the $f$-vector into the $g$-vector, which has length $\lceil{\frac{d+1}{2}}\rceil$. Thus, to determine the $f$-vector of $P$, we only need to know approximately half of its entries. This raises the question: Which $(\lceil{\frac{d+1}{2}}\rceil)$-subsets of the $f$-vector of a general simplicial polytope are sufficient to determine the whole $f$-vector? We prove that the answer is given by the bases of the Catalan matroid.
\end{abstract}

\maketitle


\bigskip

\section{Introduction}

It was conjectured by McMullen \cite{mcm} and subsequently proven by Billera--Lee, and Stanley \cite{bla, blb, stan1} that $f$-vectors of simplicial polytopes can be fully characterized by the $g$-theorem. The $g$-theorem determines whether a vector of positive integers is indeed the $f$-vector of some simplicial polytope. The Dehn--Sommerville relations condense the $f$-vector into the $g$-vector, which has length $\lceil{\frac{d+1}{2}}\rceil$. This raises the question: Which $(\lceil{\frac{d+1}{2}}\rceil)$-subsets of the $f$-vector of a general simplicial polytope are sufficient to determine the whole $f$-vector? Define a {\it Dehn--Sommerville basis} to be a minimal subset $S$ such that $\{f_i\mid i\in S\}$ determines the entire $f$-vector for any simplicial polytope. 

For example, using the $g$-theorem one can check that $f_1=(1,8,27,38,19)$ and $f_2=(1,9,28,38,19)$ are $f$-vectors of two different simplicial $4-$polytopes. Therefore the entries $f=(1,*,*,38,19)$ do not determine a simplicial $f$-vector uniquely, and $\{1,4,5\}$ is not a Dehn-Sommerville basis in dimension 4. 

In this paper we prove the following theorem.
\begin{theorem}\label{main2}
The Dehn--Sommerville bases of dimension $2n$ are precisely the upstep sets of the Dyck paths of length $2(n+1)$.
\end{theorem}
A similar description holds for the Dehn-Sommerville bases of dimension $2n-1$.

The paper is organized as follows. In Section 2 we define matroids, the Dehn--Sommerville matrix, and the Catalan matroid defined by Ardila \cite{ardila} and Bonin--de Mier--Noy \cite{bmn}. In Section 3 we prove our main result via several lemmas.

\section{Matroids and the Dehn--Sommerville matrix}
\subsection{Matroids}

A matroid is a combinatorial object that generalizes the notion of independence. We provide the definition in terms of bases; there are several other equivalent axiomatic definitions available. For a more complete study of matroids see Oxley \cite{oxley}.

\begin{definition}A {\it matroid} $\M$ is a pair $(E,\BB)$ consisting of a finite set $E$ and a nonempty collection of subsets $\BB=\BB(\M)$ of $E$, called the {\it bases} of $\M$, that satisfy the following properties:
	\begin{enumerate}
		\item[(B1)] $\emptyset \in \BB$
		\item[(B2)] (Basis exchange axiom) If $B_1,B_2\in\BB$ and $b_1\in B_1-B_2$, then there exists an element $b_2\in B_2-B_1$ such that $B_1-\{b_1\}\cup\{b_2\}\in\BB$.
	\end{enumerate} 
\end{definition}
\begin{example}\label{linear}
A key example is the matroid $M(A)$ of a matrix $A$. Let $A$ be a $d\times n$ matrix of rank $d$ over a field $K$. Denote the columns of $A$ by $\mathbf{a_1},\mathbf{a_2},\dots,\mathbf{a_n}\in K^d$. Then $B\subset[n]$ is a basis of $M(A)$ on the ground set $[n]$ if $\{\mathbf{a_i}\mid i\in B\}$ forms a linear basis for $K^d$.
\end{example}

\subsection{The Dehn-Sommerville relations} 
\begin{definition}
Let $P$ be a $d$-dimensional simplicial polytope, that is, a polytope whose facets are simplices. Then define $f(P)=(f_{-1},f_0,f_1,\dots,f_{d-1})$ to be the \textit{$f$-vector of $P$} where $f_i$ is the number of $i$-dimensional faces of $P$. It is convention that $f_{-1}=1$.
\end{definition}

\begin{definition}
For $k\in[0,d]$, the \textit{$h$-vector} of a simplicial polytope is a sequence with elements 
\begin{align*}
h_k=\sum_{i=0}^k(-1)^{k-i}{{d-i}\choose {k-i}}f_{i-1}.
\end{align*}
\end{definition}

\begin{definition}
The \textit{$g$-vector} of a simplicial polytope is a sequence where $g_0=1$ and $g_i=h_i-h_{i-1}$ for $i\in[1,\lfloor{\frac d 2}\rfloor]$.
\end{definition}

The Dehn--Sommerville relations can be stated most simply in terms of the $h$-vector.

\begin{theorem}\label{h}\cite{mcm2}
The $h$-vector of a simplicial $d$-polytope satisfies 
$$h_k=h_{d-k}$$ for $k=0,1,\dots,d$.
\end{theorem}

We now discuss a matrix reformation of Theorem \ref{h}.

\begin{definition} The {\it Dehn-Sommerville matrix},  $M_d$, is defined by
\begin{equation*}
(M_d)_{\substack{0\leq i\leq \lfloor{\frac d2}\rfloor \\ 0\leq j\leq d}}:= {d+1-i \choose d+1-j}-{i \choose d+1-j},
\end{equation*}
for $d\in\N$.

\end{definition}

\begin{theorem}\cite{bjc}
Let $P$ be a simplicial $d$-polytope, and let $f$ and $g$ denote its $f$ and $g$-vectors. Then 
\begin{align*}
g\cdot M_d = f.
\end{align*}
\end{theorem}

Label the columns of $M_d$ from 1 to $d+1$. Just as in Example \ref{linear}, we can define the {\it Dehn-Sommerville matroid} of rank $d$, $DS_d$, to be the pair $([d+1],\BB)$ where $B\in\BB$ if $B$ is a collection of columns associated with a non-zero maximal minor.

\begin{example}\label{ex1}
Let $d=4$. Then 
\begin{equation*}
M_4=\left(\begin{matrix} 1&5&10&10&5\\0&1&4&6&3\\0&0&1&2&1\end{matrix}\right)
\end{equation*}
and the bases of $DS_4$ are 
\begin{equation*}
\BB=\{123, 124, 125, 134, 135\}
\end{equation*}
where basis $ijk$ refers to the submatrix formed by columns $i,j$, and $k$ in $M_4$. 
\end{example}

\begin{definition}
Define the graph {\bf DS}$_d$ to be a triangular directed graph (or \textit{digraph}), as shown in Figure \ref{digraphex}, where all horizontal edges are directed east and all vertical edges are directed north. 
\end{definition}

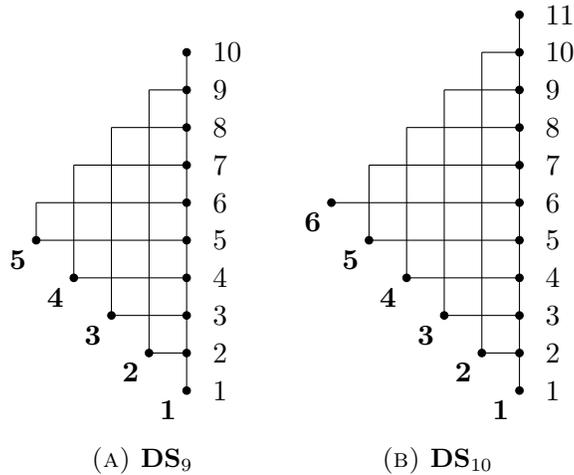
\begin{figure}[htbp]
\centering
\begin{subfigure}[b]{0.3\textwidth}
\begin{tikzpicture}

\draw[thin] (0,5/2) -- (2,5/2) ;
\draw[thin] (0,2) -- (2,2) ;
\draw[thin] (.5,3/2) -- (2,3/2) ;
\draw[thin] (1,1) -- (2,1) ;
\draw[thin] (3/2,.5) -- (2,.5) ;
\draw[thin] (2,0) -- (2,9/2) ;
\draw[thin] (.5,3) -- (2,3) ;
\draw[thin] (1,7/2) -- (2,7/2) ;
\draw[thin] (3/2,4) -- (2,4) ;
\draw[thin] (3/2,.5) -- (3/2,4) ;
\draw[thin] (1,1) -- (1,7/2) ;
\draw[thin] (.5,3/2) -- (.5,3) ;
\draw[thin] (0,2) -- (0,5/2) ;
\draw[black, fill] (2,0) circle (.05cm) node[anchor = north east]{\bf{1}} node[anchor = west]{\: 1} ;
\draw[black, fill] (3/2,.5) circle (.05cm) node[anchor = north east]{\bf{2}};
\draw[black, fill] (1,1) circle (.05cm) node[anchor = north east]{\bf{3}};
\draw[black, fill] (.5,3/2) circle (.05cm) node[anchor = north east]{\bf{4}};
\draw[black, fill] (0,2) circle (.05cm) node[anchor = north east]{\bf{5}};
\draw[black, fill] (2,1/2) circle (.05cm) node[anchor = west]{\: 2};
\draw[black, fill] (2,1) circle (.05cm) node[anchor = west]{\: 3};
\draw[black, fill] (2,3/2) circle (.05cm) node[anchor = west]{\: 4};
\draw[black, fill] (2,2) circle (.05cm) node[anchor = west]{\: 5};
\draw[black, fill] (2,5/2) circle (.05cm) node[anchor = west]{\: 6};
\draw[black, fill] (2,3) circle (.05cm) node[anchor = west]{\: 7};
\draw[black, fill] (2,7/2) circle (.05cm) node[anchor = west]{\: 8};
\draw[black, fill] (2,4) circle (.05cm) node[anchor = west]{\: 9};
\draw[black, fill] (2,9/2) circle (.05cm) node[anchor = west]{\: 10};

\end{tikzpicture}
 \caption{{\bf DS}$_9$}
      
\end{subfigure}
~
\begin{subfigure}[b]{0.3\textwidth}
\begin{tikzpicture}

\draw[thin] (0,5/2) -- (5/2,5/2) ;
\draw[thin] (1/2,2) -- (5/2,2) ;
\draw[thin] (1,3/2) -- (5/2,3/2) ;
\draw[thin] (3/2,1) -- (5/2,1) ;
\draw[thin] (2,1/2) -- (5/2,1/2) ;
\draw[thin] (5/2,0) -- (5/2,5) ;
\draw[thin] (1/2,3) -- (5/2,3) ;
\draw[thin] (1,7/2) -- (5/2,7/2) ;
\draw[thin] (3/2,4) -- (5/2,4) ;
\draw[thin] (2,9/2) -- (5/2,9/2) ;
\draw[thin] (2,1/2) -- (2,9/2) ;
\draw[thin] (3/2,1) -- (3/2,4) ;
\draw[thin] (1,3/2) -- (1,7/2) ;
\draw[thin] (1/2,2) -- (1/2,3) ;
\draw[black, fill] (5/2,0) circle (.05cm) node[anchor = north east]{\bf{1}} node[anchor = west]{\: 1} ;
\draw[black, fill] (2,1/2) circle (.05cm) node[anchor = north east]{ \bf{2}};
\draw[black, fill] (3/2,1) circle (.05cm) node[anchor = north east]{\bf{3}};
\draw[black, fill] (1,3/2) circle (.05cm) node[anchor = north east]{\bf{4}};
\draw[black, fill] (1/2,2) circle (.05cm) node[anchor = north east]{\bf{5}};
\draw[black, fill] (0,5/2) circle (.05cm) node[anchor = north east]{\bf{6}};
\draw[black, fill] (5/2,1/2) circle (.05cm) node[anchor = west]{\: 2};
\draw[black, fill] (5/2,1) circle (.05cm) node[anchor = west]{\: 3};
\draw[black, fill] (5/2,3/2) circle (.05cm) node[anchor = west]{\: 4};
\draw[black, fill] (5/2,2) circle (.05cm) node[anchor = west]{\: 5};
\draw[black, fill] (5/2,5/2) circle (.05cm) node[anchor = west]{\: 6};
\draw[black, fill] (5/2,3) circle (.05cm) node[anchor = west]{\: 7};
\draw[black, fill] (5/2,7/2) circle (.05cm) node[anchor = west]{\: 8};
\draw[black, fill] (5/2,4) circle (.05cm) node[anchor = west]{\: 9};
\draw[black, fill] (5/2,9/2) circle (.05cm) node[anchor = west]{\: 10};
\draw[black, fill] (5/2,5) circle (.05cm) node[anchor = west]{\: 11};
\end{tikzpicture}

        \caption{{\bf DS}$_{10}$}
      
    \end{subfigure}

 \caption{Dehn--Sommerville graphs for odd and even $d$.}\label{digraphex}
\end{figure}

\begin{definition}
Number the nodes along the southwest border $1,\dots, \lceil{\frac{d+1}{2}}\rceil$ in bold, and call them \emph{sources}. Number the nodes along the east border $1,\dots,d+1$, and call them \emph{sinks}.  A \textit{routing} is a set of vertex-disjoint paths on a digraph.
\end{definition}

Considering the collection of routings on {\bf DS}$_d$ produces the following theorem.

\begin{theorem}\label{routing}
A subset $B \subset[d+1]$ is a basis of $DS_d$ if and only if there is a routing from the source set $[\mathbf{\lceil{\frac{d+1}{2}}\rceil}]$ to the sink set $B$ in the graph {\bf DS}$_d$. 
\end{theorem}

\begin{proof}
Bj\"{o}rklund and Engst\"om \cite{eng} found a way to give positive weights to the edges of {\bf DS}$_d$ so that $M_d$ is the path matrix of {\bf DS}$_d$; that is, entry $(M_d)_{ij}$ equals the sum of the product of the weights of the paths from source $\mathbf{i}$ to sink $j$. Then, by the Lindstr\"{o}m--Gessel--Viennot lemma, the determinant of columns $\mathbf{a_{i_1}},\dots,\mathbf{a_{i_{\lceil{\frac{d+1}{2}}\rceil}}}$ is the sum of the product of weights of the routing from the source nodes $\mathbf{i_1<\dots<i_{\lceil{\frac{d+1}{2}}\rceil}}$ to the sink nodes  $a_{i_1},\dots,a_{i_{\lceil{\frac{d+1}{2}}\rceil}}$. It follows that the determinant is non-zero if and only if there is a routing.
\end{proof}

\begin{example}\label{routings}
Continuing with $d=4$, we see the destination sets of the routings on digraph {\bf DS}$_4$ match with the bases of $DS_4$ found in Example \ref{ex1}:
    \begin{center}
    \begin{tabular}{c}
        \begin{tikzpicture}
            \draw[black, ultra thick] (.5,.5) -- (1,.5);
            \draw[black, ultra thick] (0,1) -- (1,1);
            \draw[thin] (.5,.5) -- (.5,1.5) ;
            \draw[thin] (.5,1.5) -- (1,1.5) ;
            \draw[thin] (1,0) -- (1,2) ;
            \draw[thin] (.5,1.5) -- (1,1.5) ;
            
            \draw[black, fill] (1,0) circle (.05cm) node[anchor = north east]{\bf{1}} node[anchor = west]{\:1} ;
            \draw[black, fill] (.5,.5) circle (.05cm) node[anchor = north east]{\bf{2}};
            \draw[black, fill] (0,1) circle (.05cm) node[anchor = north east]{\bf{3}};
            \draw[black, fill] (1,.5) circle (.05cm) node[anchor = west]{\:2};
            \draw[black, fill] (1,1) circle (.05cm) node[anchor = west]{\:3};
            \draw[black, fill] (1,3/2) circle (.05cm) node[anchor = west]{\:4};
            \draw[black, fill] (1,2) circle (.05cm) node[anchor = west]{\:5};
            \draw[black] (1,1) circle (.1cm);
            \draw[black] (1,.5) circle (.1cm);
            \draw[black] (1,0) circle (.1cm);
        \end{tikzpicture}
        
         \begin{tikzpicture}
            \draw[black, ultra thick] (.5,.5) -- (1,.5);
            \draw[black, ultra thick] (0,1) -- (1,1);
            \draw[black, ultra thick] (1,1) -- (1,1.5);
            \draw[thin] (.5,.5) -- (.5,1.5) ;
            \draw[thin] (.5,1.5) -- (1,1.5) ;
            \draw[thin] (1,0) -- (1,1) ;
            \draw[thin] (1,1.5) -- (1,2) ;
            \draw[thin] (.5,1.5) -- (1,1.5) ;
            
            \draw[black, fill] (1,0) circle (.05cm) node[anchor = north east]{\bf{1}} node[anchor = west]{\:1} ;
            \draw[black, fill] (.5,.5) circle (.05cm) node[anchor = north east]{\bf{2}};
            \draw[black, fill] (0,1) circle (.05cm) node[anchor = north east]{\bf{3}};
            \draw[black, fill] (1,.5) circle (.05cm) node[anchor = west]{\:2};
            \draw[black, fill] (1,1) circle (.05cm) node[anchor = west]{\:3};
            \draw[black, fill] (1,3/2) circle (.05cm) node[anchor = west]{\:4};
            \draw[black, fill] (1,2) circle (.05cm) node[anchor = west]{\:5};
            \draw[black] (1,1.5) circle (.1cm);
            \draw[black] (1,.5) circle (.1cm);
            \draw[black] (1,0) circle (.1cm);
        \end{tikzpicture}
        
        \begin{tikzpicture}
            \draw[black, ultra thick] (.5,.5) -- (1,.5);
            \draw[black, ultra thick] (0,1) -- (1,1);
            \draw[black, ultra thick] (1,1) -- (1,2);
            \draw[thin] (.5,.5) -- (.5,1.5) ;
            \draw[thin] (.5,1.5) -- (1,1.5) ;
            \draw[thin] (1,0) -- (1,1) ;
            \draw[thin] (.5,1.5) -- (1,1.5) ;
            
            \draw[black, fill] (1,0) circle (.05cm) node[anchor = north east]{\bf{1}} node[anchor = west]{\:1} ;
            \draw[black, fill] (.5,.5) circle (.05cm) node[anchor = north east]{\bf{2}};
            \draw[black, fill] (0,1) circle (.05cm) node[anchor = north east]{\bf{3}};
            \draw[black, fill] (1,.5) circle (.05cm) node[anchor = west]{\:2};
            \draw[black, fill] (1,1) circle (.05cm) node[anchor = west]{\:3};
            \draw[black, fill] (1,3/2) circle (.05cm) node[anchor = west]{\:4};
            \draw[black, fill] (1,2) circle (.05cm) node[anchor = west]{\:5};
            \draw[black] (1,2) circle (.1cm);
            \draw[black] (1,.5) circle (.1cm);
            \draw[black] (1,0) circle (.1cm);
        \end{tikzpicture}
        
        \begin{tikzpicture}
            \draw[black, ultra thick] (.5,.5) -- (1,.5);
            \draw[black, ultra thick] (1,.5) -- (1,1);
            \draw[black, ultra thick] (0,1) -- (.5,1);
            \draw[black, ultra thick] (.5,1) -- (.5,1.5);
            \draw[black, ultra thick] (.5,1.5) -- (1,1.5);
            \draw[thin] (.5,.5) -- (.5,1) ;
            \draw[thin] (.5,1) -- (1,1) ;
            \draw[thin] (.5,1.5) -- (1,1.5) ;
            \draw[thin] (1,0) -- (1,.5) ;
            \draw[thin] (1,1.5) -- (1,2) ;
            \draw[thin] (1,1) -- (1,1.5) ;
            
            \draw[black, fill] (1,0) circle (.05cm) node[anchor = north east]{\bf{1}} node[anchor = west]{\:1} ;
            \draw[black, fill] (.5,.5) circle (.05cm) node[anchor = north east]{\bf{2}};
            \draw[black, fill] (0,1) circle (.05cm) node[anchor = north east]{\bf{3}};
            \draw[black, fill] (1,.5) circle (.05cm) node[anchor = west]{\:2};
            \draw[black, fill] (1,1) circle (.05cm) node[anchor = west]{\:3};
            \draw[black, fill] (1,3/2) circle (.05cm) node[anchor = west]{\:4};
            \draw[black, fill] (1,2) circle (.05cm) node[anchor = west]{\:5};
            \draw[black] (1,1.5) circle (.1cm);
            \draw[black] (1,1) circle (.1cm);
            \draw[black] (1,0) circle (.1cm);
        \end{tikzpicture}
        
         \begin{tikzpicture}
            \draw[black, ultra thick] (.5,.5) -- (1,.5);
            \draw[black, ultra thick] (1,.5) -- (1,1);
            \draw[black, ultra thick] (0,1) -- (.5,1);
            \draw[black, ultra thick] (.5,1) -- (.5,1.5);
            \draw[black, ultra thick] (.5,1.5) -- (1,1.5);
            \draw[black, ultra thick] (1,1.5) -- (1,2) ;
            \draw[thin] (.5,.5) -- (.5,1) ;
            \draw[thin] (.5,1) -- (1,1) ;
            \draw[thin] (.5,1.5) -- (1,1.5) ;
            \draw[thin] (1,0) -- (1,.5) ;
            \draw[thin] (1,1) -- (1,1.5) ;
            
            \draw[black, fill] (1,0) circle (.05cm) node[anchor = north east]{\bf{1}} node[anchor = west]{\:1} ;
            \draw[black, fill] (.5,.5) circle (.05cm) node[anchor = north east]{\bf{2}};
            \draw[black, fill] (0,1) circle (.05cm) node[anchor = north east]{\bf{3}};
            \draw[black, fill] (1,.5) circle (.05cm) node[anchor = west]{\:2};
            \draw[black, fill] (1,1) circle (.05cm) node[anchor = west]{\:3};
            \draw[black, fill] (1,3/2) circle (.05cm) node[anchor = west]{\:4};
            \draw[black, fill] (1,2) circle (.05cm) node[anchor = west]{\:5};
            \draw[black] (1,2) circle (.1cm);
            \draw[black] (1,1) circle (.1cm);
            \draw[black] (1,0) circle (.1cm);
        \end{tikzpicture}
        
    \end{tabular}
    \end{center}
\end{example}

\subsection{The Catalan Matroid} 


\begin{definition}
For $n\in\N$, a {\it Dyck path} of length $2n$ is a path in the plane from $(0,0)$ to $(2n,0)$ with upsteps, $(1,1)$, and downsteps, $(1,-1)$, that never falls below the $x$-axis. The number of Dyck paths of length $2n$ is counted by the Catalan numbers $C_n = \frac{1}{n+1}{2n \choose n}$.
\end{definition}
\begin{definition}\cite{ardila, bmn} The {\it Catalan matroid}, $\mathbf{C}_n$, is a matroid with a ground set of $[2n]$ whose bases are the upstep sets of the Dyck paths of length $2n$. 
\end{definition}
\begin{example}\label{dyckex}
Let $n=3$. Then the bases of the Catalan matroid are the upstep sets of the following Dyck paths of length 6:
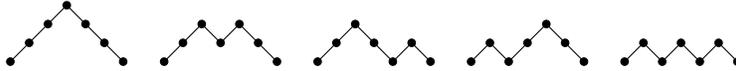
\begin{center}
\begin{tabular}{c}
\begin{tikzpicture}
\draw[thin] (0,0) -- (1/4,1/4) ;
\draw[thin] (1/4,1/4) -- (1/2,1/2) ;
\draw[thin] (1/2,1/2) -- (3/4,3/4) ;
\draw[thin] (3/4,3/4) -- (1,1/2) ;
\draw[thin] (1,1/2) -- (5/4,1/4) ;
\draw[thin] (5/4,1/4) -- (3/2,0) ;
\draw[black, fill] (0,0) circle (.05cm);
\draw[black, fill] (1/4,1/4) circle (.05cm);
\draw[black, fill] (1/2,1/2) circle (.05cm);
\draw[black, fill] (3/4,3/4) circle (.05cm);
\draw[black, fill] (1,1/2) circle (.05cm);
\draw[black, fill] (5/4,1/4) circle (.05cm);
\draw[black, fill] (3/2,0) circle (.05cm);
\end{tikzpicture}
\:\:
\begin{tikzpicture}
\draw[thin] (0,0) -- (1/4,1/4) ;
\draw[thin] (1/4,1/4) -- (1/2,1/2) ;
\draw[thin] (1/2,1/2) -- (3/4,1/4) ;
\draw[thin] (3/4,1/4) -- (1,1/2) ;
\draw[thin] (1,1/2) -- (5/4,1/4) ;
\draw[thin] (5/4,1/4) -- (3/2,0) ;
\draw[black, fill] (0,0) circle (.05cm);
\draw[black, fill] (1/4,1/4) circle (.05cm);
\draw[black, fill] (1/2,1/2) circle (.05cm);
\draw[black, fill] (3/4,1/4) circle (.05cm);
\draw[black, fill] (1,1/2) circle (.05cm);
\draw[black, fill] (5/4,1/4) circle (.05cm);
\draw[black, fill] (3/2,0) circle (.05cm);
\end{tikzpicture}
\:\:
\begin{tikzpicture}
\draw[thin] (0,0) -- (1/4,1/4) ;
\draw[thin] (1/4,1/4) -- (1/2,1/2) ;
\draw[thin] (1/2,1/2) -- (3/4,1/4) ;
\draw[thin] (3/4,1/4) -- (1,0) ;
\draw[thin] (1,0) -- (5/4,1/4) ;
\draw[thin] (5/4,1/4) -- (3/2,0) ;
\draw[black, fill] (0,0) circle (.05cm);
\draw[black, fill] (1/4,1/4) circle (.05cm);
\draw[black, fill] (1/2,1/2) circle (.05cm);
\draw[black, fill] (3/4,1/4) circle (.05cm);
\draw[black, fill] (1,0) circle (.05cm);
\draw[black, fill] (5/4,1/4) circle (.05cm);
\draw[black, fill] (3/2,0) circle (.05cm);
\end{tikzpicture}
\:\:
\begin{tikzpicture}
\draw[thin] (0,0) -- (1/4,1/4) ;
\draw[thin] (1/4,1/4) -- (1/2,0) ;
\draw[thin] (1/2,0) -- (3/4,1/4) ;
\draw[thin] (3/4,1/4) -- (1,1/2) ;
\draw[thin] (1,1/2) -- (5/4,1/4) ;
\draw[thin] (5/4,1/4) -- (3/2,0) ;
\draw[black, fill] (0,0) circle (.05cm);
\draw[black, fill] (1/4,1/4) circle (.05cm);
\draw[black, fill] (1/2,0) circle (.05cm);
\draw[black, fill] (3/4,1/4) circle (.05cm);
\draw[black, fill] (1,1/2) circle (.05cm);
\draw[black, fill] (5/4,1/4) circle (.05cm);
\draw[black, fill] (3/2,0) circle (.05cm);
\end{tikzpicture}
\:\:
\begin{tikzpicture}
\draw[thin] (0,0) -- (1/4,1/4) ;
\draw[thin] (1/4,1/4) -- (1/2,0) ;
\draw[thin] (1/2,0) -- (3/4,1/4) ;
\draw[thin] (3/4,1/4) -- (1,0) ;
\draw[thin] (1,0) -- (5/4,1/4) ;
\draw[thin] (5/4,1/4) -- (3/2,0) ;
\draw[black, fill] (0,0) circle (.05cm);
\draw[black, fill] (1/4,1/4) circle (.05cm);
\draw[black, fill] (1/2,0) circle (.05cm);
\draw[black, fill] (3/4,1/4) circle (.05cm);
\draw[black, fill] (1,0) circle (.05cm);
\draw[black, fill] (5/4,1/4) circle (.05cm);
\draw[black, fill] (3/2,0) circle (.05cm);
\end{tikzpicture}
\end{tabular}
\end{center}

\begin{equation*}
\BB=\{123, 124, 125, 134, 135\}
\end{equation*}

\end{example}

\section{Main Result}

\begin{theorem}\label{main}
The Dehn--Sommerville matroids are obtained from the Catalan matroids by removing trivial elements:
\begin{align*}
DS_{2n} \cong \mathbf{C}_{n+1}\backslash (2n+2) \text{ \; and \; }
DS_{2n-1} \cong \mathbf{C}_{n+1}\backslash 1\backslash (2n+2)\;\; \text{ for } n\in\N.
\end{align*}
Note that 1 is a coloop and $2n+2$ is a loop in $\mathbf{C}_{n+1}$.
\end{theorem}

To prove our main result, we first set the stage with some preliminary lemmas. We first show there is a close connection between odd and even Dehn--Sommerville matroids. We  then characterize the bases of the Dehn--Sommerville and Catalan matroids, respectively. Finally, we prove our main result.

\begin{lemma}\label{bij}
The bases of the Dehn--Sommerville matroid $DS_{2n}$ are in bijection with the bases of the Dehn--Sommerville matroid $DS_{2n-1}$. More precisely, $DS_{2n}\cong DS_{2n-1}\oplus {C}$, where $C$ is a coloop.
\end{lemma}
\begin{proof}
For the graph {\bf DS}$_{2n}$, any path leaving the top source node $\mathbf{\lceil\frac{2n+1}{2}\rceil} = n+1$ must first travel east along the only edge leaving it. To satisfy the condition that all paths in a routing are vertex-disjoint, any path leaving all other source nodes, excluding {\bf 1}, must also first travel east. Source node {\bf 1} is forced to have the trivial path. Thus, the horizontal edges from every source node, excluding {\bf 1}, can be contracted without affecting the potential destinations of the routings. This results in a graph equal to the graph {\bf DS}$_{2n-1}$ with an added vertical edge extending from the bottom.

It follows that the number of bases of $DS_{2n}$ is equal to the number of bases of $DS_{2n-1}$ and there exists a bijection between the bases; namely,
\begin{align*}
\{1,b_2,b_3,\dots,b_n\}\mapsto\{b_2-1,b_3-1,\dots,b_n-1\},
\end{align*}
where $\{1,b_2,b_3,\dots,b_n\}$ is a basis of $DS_{2n}$. 
\end{proof}

\begin{lemma}\label{catalan}\cite[Problem 6.19 $(i,t)$]{stan2}
Let $a_1 < a_2 < \dots < a_n$ be the upsteps of a lattice path $P$ of length $2n$. Then $P$ is a Dyck path if and only if  $$a_1 = 1, a_2 \leq 3, a_4 \leq 5, \dots, a_n \leq 2n-1.$$
\end{lemma}

\begin{lemma}\label{ds}
Let $B=\{b_1 < b_2 < \dots < b_{n+1}\}$ be a subset of $[2n+1]$. Then $B$ is a basis of $DS_{2n}$ if and only if $b_1 =1, b_2 \leq 3, b_3\leq 5, \dots, b_{n+1} \leq 2n+1$. 
\end{lemma}

\begin{proof}
First let us prove the forward direction. Assume $B=\{b_1 < b_2 < \dots < b_{n+1}\}$ is a basis of $DS_{2n}$. Assume $b_i\geq 2i$. Notice that the northwest diagonal starting at sink $2i$ has $n-i$ vertices (excluding sink $2i$). If there was a routing, the $n+1-i$ paths starting at source nodes $\mathbf{i+1, \dots, n+1}$ would pass through this northwest diagonal, which forms a bottleneck of width $n-i$, contradicting the presence of a routing. Thus, assuming a routing exists from sources $[\mathbf{i}]$ to sinks $b_1,\dots, b_i$, the remaining source nodes $\mathbf{i+1,\dots, n+1}$ can only be routed if $b_i\leq 2i-1$. See Figure \ref{break} for an example of a sink set that creates a bottleneck when $2n=10$ and $b_4=8$.

To prove the other direction, assume $b_1 =1, b_2 \leq 3, b_3\leq 5, \dots, b_{n+1} \leq 2n+1$. Construct a routing sequentially from the bottom, making each path as low as possible. All routings must begin with $b_1=1$. For each $i$, we need to check whether the number of nodes remaining on the northwest diagonal containing $b_i$ is greater than or equal to the number of remaining sink nodes. Since $b_i\leq 2i-1$, then by our previous argument no bottlenecks are formed and this gives us a valid routing, so $B=\{b_1 < b_2 < \dots < b_{n+1}\}$ is a basis of $DS_{2n}$.
\end{proof}

\begin{proof}[Proof of Theorem \ref{main}]
This follows immediately from Lemmas \ref{bij}, \ref{catalan}, and \ref{ds}.
\end{proof}

Note that Theorem \ref{main2} describes in words the bases of the matroids involved in Theorem \ref{main}.

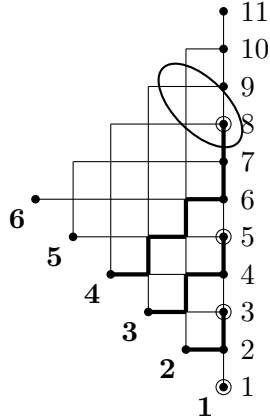
\begin{figure}[htbp]
\begin{center}

\begin{tikzpicture}

\draw[ultra thick] (2,1/2) -- (2.5,1/2) ;
\draw[ultra thick] (2.5,1/2) -- (2.5,1) ;
\draw[ultra thick] (1.5,1) -- (2,1) ;
\draw[ultra thick] (2,1) -- (2,1.5) ;
\draw[ultra thick] (2,1.5) -- (2.5,1.5) ;
\draw[ultra thick] (2.5,1.5) -- (2.5,2) ;

\draw[ultra thick] (1,1.5) -- (1.5,1.5) ;
\draw[ultra thick] (1.5,1.5) -- (1.5,2) ;
\draw[ultra thick] (1.5, 2) -- (2,2) ;
\draw[ultra thick] (2,2) -- (2,2.5) ;
\draw[ultra thick] (2,2.5) -- (2.5,2.5) ;
\draw[ultra thick] (2.5,2.5) -- (2.5,3.5) ;

\draw[black,thick, rotate=45] (4.2,1.1) ellipse (10pt and 20pt);


\draw[thin] (0,5/2) -- (5/2,5/2) ;
\draw[thin] (1/2,2) -- (5/2,2) ;
\draw[thin] (1,3/2) -- (5/2,3/2) ;
\draw[thin] (3/2,1) -- (5/2,1) ;
\draw[thin] (2,1/2) -- (5/2,1/2) ;
\draw[thin] (5/2,0) -- (5/2,5) ;
\draw[thin] (1/2,3) -- (5/2,3) ;
\draw[thin] (1,7/2) -- (5/2,7/2) ;
\draw[thin] (3/2,4) -- (5/2,4) ;
\draw[thin] (2,9/2) -- (5/2,9/2) ;
\draw[thin] (2,1/2) -- (2,9/2) ;
\draw[thin] (3/2,1) -- (3/2,4) ;
\draw[thin] (1,3/2) -- (1,7/2) ;
\draw[thin] (1/2,2) -- (1/2,3) ;
\draw[black] (2.5,0) circle (.1cm);
\draw[black] (2.5,1) circle (.1cm);
\draw[black] (2.5,2) circle (.1cm);
\draw[black] (2.5,3.5) circle (.1cm);
\draw[black, fill] (5/2,0) circle (.05cm) node[anchor = north east]{\bf{1}} node[anchor = west]{\:1} ;
\draw[black, fill] (2,1/2) circle (.05cm) node[anchor = north east]{\bf{2}};
\draw[black, fill] (3/2,1) circle (.05cm) node[anchor = north east]{\bf{3}};
\draw[black, fill] (1,3/2) circle (.05cm) node[anchor = north east]{\bf{4}};
\draw[black, fill] (1/2,2) circle (.05cm) node[anchor = north east]{\bf{5}};
\draw[black, fill] (0,5/2) circle (.05cm) node[anchor = north east]{\bf{6}};
\draw[black, fill] (5/2,1/2) circle (.05cm) node[anchor = west]{\:2};
\draw[black, fill] (5/2,1) circle (.05cm) node[anchor = west]{\:3};
\draw[black, fill] (5/2,3/2) circle (.05cm) node[anchor = west]{\:4};
\draw[black, fill] (5/2,2) circle (.05cm) node[anchor = west]{\:5};
\draw[black, fill] (5/2,5/2) circle (.05cm) node[anchor = west]{\:6};
\draw[black, fill] (5/2,3) circle (.05cm) node[anchor = west]{\:7};
\draw[black, fill] (5/2,7/2) circle (.05cm) node[anchor = west]{\:8};
\draw[black, fill] (5/2,4) circle (.05cm) node[anchor = west]{\:9};
\draw[black, fill] (5/2,9/2) circle (.05cm) node[anchor = west]{\:10};
\draw[black, fill] (5/2,5) circle (.05cm) node[anchor = west]{\:11};
\end{tikzpicture}
\end{center}
\caption{Bottleneck prevents a routing for sinks (1,3,5,8,*,*) for $d = 2n = 10$.\label{break}}
\end{figure}

\noindent \textbf{Remark.} Bj\"{o}rner \cite{bjd} conjectured and Bj\"{o}rklund and Engst\"om \cite{eng} proved that $M_d$ is totally non-negative, which implies $DS_d$ is a positroid. A corollary to our main theorem is that the Catalan matroid is a positroid, a result recently discovered by Pawlowski \cite{paw} via pattern avoiding permutations. Thus, there are several positroid representations introduced by Postnikov \cite{pos} that can be used to further represent the Catalan matroid. 
Using Lemma \ref{ds}, one can verify that the decorated permutation corresponding to the positroid $\mathbf{C}_{n+1}$ is

\begin{align*}
   \bar{1}35\cdots (2n+1)246\cdots (2n)  \:\:\:  & \text{when} \:\:\:DS_{2n} \cong \mathbf{C}_{n+1}\backslash (2n+2)\\
     246 \cdots (2n)135\cdots (2n-1)  \:\:\: & \text{when} \:\:\:  DS_{2n-1} \cong \mathbf{C}_{n+1}\backslash 1\backslash (2n+2) \\
\end{align*} 
for $n\in\N$.

We can use this relation to obtain the Grassmann necklace, Le diagram and juggling pattern corresponding to $\mathbf{C}_{n+1}$.

\smallskip

\noindent \textbf{Acknowledgments.} This work is part of Anastasia Chavez's Ph.D. thesis at the University of California, Berkeley and Nicole Yamzon's Master's thesis at San Francisco State University. We are very thankful to our adviser, Dr.~Federico Ardila for his guidance throughout this project. Thank you to Dr.~Lauren Williams for the helpful conversations.

\nocite{*}
\bibliographystyle{plain}
\bibliography{DSCatalan.bib}

\begin{thebibliography}{10}

\bibitem{ardila}
F.~Ardila.
\newblock {T}he {C}atalan {M}atroid.
\newblock {\em J Comb Theory A}, 104:49--62, 2003.

\bibitem{arw}
F.~Ardila, F.~Rinc{\'{o}}n, and L.~Williams.
\newblock {P}ositroids and {N}on-{C}rossing {P}artitions.
\newblock {\em to appear in T. Am. Math. Soc.}

\bibitem{bla}
L.~Billera and C.~Lee.
\newblock {S}ufficiency of {M}cmullen's {C}onditions for $f$-vectors of
  {S}implicial {P}olytopes.
\newblock {\em Bull. Amer. Math. Soc.}, 2:181--185, 1980.

\bibitem{blb}
L.~Billera and C.~Lee.
\newblock {A} {P}roof of the {S}ufficiency of {M}cmullen's {C}onditions for
  $f$-vectors of {S}implicial {P}olytopes.
\newblock {\em J. Combin. Theory Ser. A}, 31:237--255, 1981.

\bibitem{eng}
M.~Bj\"{o}rklund and A.~Engst\"om.
\newblock {T}he $g$-{T}heorem {M}atrices are {T}otally {N}onnegative.
\newblock {\em J. Combin. Theory Ser. A}, 116:730--732, 2009.

\bibitem{bja}
A.~Bj{\"{o}}rner.
\newblock {T}he {U}nimodality {C}onjecture for {C}onvex {P}olytopes.
\newblock {\em Bull. Amer. Math. Soc.}, 4:187--188, 1981.

\bibitem{bjb}
A.~Bj{\"{o}}rner.
\newblock {F}ace {N}umbers of {C}omplexes and {P}olytopes.
\newblock {\em Proc. Intern. Congr. Math., Berkeley, CA}, 1408-1418, 1986.

\bibitem{bjc}
A.~Bj{\"{o}}rner.
\newblock {P}artial {U}nimodality for $f$-vectors of {S}implicial {P}olytopes
  and {S}pheres.
\newblock {\em ``Jerusalem Comb. '93" (H.~Barcelo and G.~Kalai, eds.), Contemp.
  Math.}, 178:45--54, 1994.
\newblock \emph{Amer. Math. Soc.}

\bibitem{bjd}
A.~Bj{\"{o}}rner.
\newblock {A} {C}omparison {T}heorem for $f$-vectors of {S}implicial
  {P}olytopes.
\newblock {\em Pure Appl. Math. Q}, 3 (1):347--356, 2007.

\bibitem{bmn}
J.~Bonin, A.~de~Mier, and M.~Noy.
\newblock {L}attice {P}ath {M}atroids: {E}numerative {A}spects and {T}utte
  {P}olynomials.
\newblock {\em J. Combin. Theory Ser. A}, 104:63--94, 2003.

\bibitem{mcm2}
P.~McMullen.
\newblock {T}he {M}aximum {N}umbers of {F}aces of a {C}onvex {P}olytope.
\newblock {\em Mathematika}, 17:179--184, 1970.

\bibitem{mcm}
P.~McMullen.
\newblock {T}he {N}umbers of {F}aces of {S}implicial {P}olytopes.
\newblock {\em Israel J. Math.}, 9:559--570, 1971.

\bibitem{oxley}
J.~Oxley.
\newblock {\em {M}atroid {T}heory}.
\newblock Oxford University Press Inc., 2nd edition, 1992.

\bibitem{paw}
B.~Pawlowski.
\newblock {C}atalan {M}atroid {D}ecompositions of {C}ertain {P}ositroids.
\newblock {\em arXiv:1502.00158 [math.CO]}.

\bibitem{pos}
A.~Postnikov.
\newblock {T}otal {P}ositivity, {G}rassmannians, and {N}etworks.
\newblock {\em Preprint. Available at
  http://www-math.mit.edu/$\sim$apost/papers/tpgrass.pdf}.

\bibitem{stan1}
R.~Stanley.
\newblock {T}he {U}pper {B}ound {C}onjecture and {C}ohen-{M}acaulay {R}ings.
\newblock {\em Stud. Appl. Math.}, 54:135--142, 1975.

\bibitem{stan2}
R.~Stanley.
\newblock {\em Enumerative Combinatorics}, volume~2.
\newblock Cambridge University Press, 1999.

\bibitem{stan}
R.~Stanley.
\newblock {\em Enumerative Combinatorics}, volume~1.
\newblock Cambridge University Press, 2012.

\bibitem{z}
G.~Ziegler.
\newblock {\em {L}ectures on {P}olytopes}.
\newblock Springer, 7th edition, 2006.

\end{thebibliography}
\label{sec:biblio}

\end{document}